\documentclass[11pt]{amsart}
\usepackage{amsmath,amssymb,mathrsfs}
\usepackage[shortlabels]{enumitem}
\usepackage[margin=1in]{geometry}
\usepackage{microtype}
\usepackage{xcolor}
\usepackage[colorlinks=true,linkcolor=blue,citecolor=blue,urlcolor=blue]{hyperref}
\numberwithin{equation}{section}
\newtheorem{theorem}{Theorem}[section]
\newtheorem{maintheorem}[theorem]{Theorem}
\newtheorem{proposition}[theorem]{Proposition}
\newtheorem{lemma}[theorem]{Lemma}

\theoremstyle{definition}
\newtheorem{conjecture}[theorem]{Conjecture}

\newtheorem{remark}[theorem]{Remark}
\newcommand{\dd}{\mathrm d}
\newcommand{\dstar}{\mathrm d^*}

\newcommand{\Ric}{\operatorname{Ric}}
\newcommand{\rank}{\operatorname{rank}}
\newcommand{\vol}{\operatorname{vol}}
\newcommand{\id}{\operatorname{id}}

\newcommand{\Hess}{\nabla^2}
\newcommand{\R}{\mathcal R}
\newcommand{\Q}{\mathcal Q}
\newcommand{\ip}[2]{\left\langle #1,#2\right\rangle}
\newcommand{\norm}[1]{\left|#1\right|}
\newcommand{\wt}{\widetilde}

\newcommand{\UW}{\operatorname{UW}}
\begin{document}
\title[Rigidity and flexibility under PIC]{Rigidity and flexibility under positive isotropic curvature}
\author{Tsz-Kiu Aaron Chow}
\address{Department of Mathematics, Hong Kong University of Science and Technology, Hong Kong S.A.R., China}
\email{\href{mailto:chowtka@ust.hk}{chowtka@ust.hk}}
\thanks{T.-K. A. C. is supported  by the Croucher Foundation Start-up Grant and the HKUST New Faculty Start-up Grant. The authors would like to thank Professor Simon Brendle for inspiring discussions.}

\author{Yipeng Wang}
\address{Department of Mathematics, Princeton University, Princeton, NJ 08544, USA}
\email{\href{mailto:yipeng.wang@princeton.edu}{yipeng.wang@princeton.edu}}

\begin{abstract}
For every $n\ge4$ and $L>0$, we construct a smooth $4$-PIC metric on
$S^n$ with Urysohn $1$-width at least $L$ and an embedded stable minimal
disk of intrinsic inradius at least $L$. These examples disprove the
proposed width and stable-disk radius bounds under a positive lower
bound for isotropic curvature. On closed even-dimensional manifolds, we prove the sharp
estimate $\lambda_1^{(2)}\ge(n-1)\sigma/2$ under $\sigma$-PIC and show
that equality forces roundness if a closed eigenform attaining the
bound has rank at least four at some point.
\end{abstract}
\maketitle
\section{Introduction}

A Riemannian manifold $(M^n,g)$, $n\ge4$, has positive isotropic curvature
(PIC) if
\[
R_{1313}+R_{1414}+R_{2323}+R_{2424}-2R_{1234}>0
\]
for every orthonormal four-frame $\{e_1,e_2,e_3,e_4\}$. We say that $g$ is $\sigma$-PIC if the
left-hand side is bounded below by $\sigma>0$, using the convention that
the unit sphere has sectional curvature one. The unit round sphere $S^n$
and the standard cylinder $S^{n-1}\times\mathbb R$ are basic examples,
which are $4$-PIC and $2$-PIC, respectively. In this
paper, we investigate both the rigidity and the flexibility of the PIC
condition.

Positive isotropic curvature was introduced by Micallef--Moore
\cite{MM1988}, who proved that a closed simply connected PIC manifold is
homeomorphic to a sphere. Micallef--Wang \cite{MW1993} subsequently showed
that PIC is preserved under connected sums. In particular, connected sums
of copies of $S^{n-1}\times S^1$ admit PIC metrics. These examples led to
the following conjectures concerning the topology of closed PIC manifolds
\cite{Gromov1996,Schoen2007}.

\begin{conjecture}[Gromov; Schoen]\label{conj:topology}
Let $M^n$ be a closed connected manifold admitting a PIC metric. Then
$\pi_1(M)$ is virtually free. More strongly, a finite cover of $M$ is
diffeomorphic to $S^n$ or to a connected sum of finitely many copies of
$S^{n-1}\times S^1$.
\end{conjecture}

This proposed classification parallels that of closed orientable
three-manifolds with positive scalar curvature. In dimension three,
quantitative geometric estimates provide a route to the topological
restrictions. Gromov--Lawson \cite[\S10]{GL1983} used radius estimates for
stable minimal surfaces to control the filling of curves. Their estimates
imply that the universal cover of a closed three-manifold with scalar
curvature at least one has universally bounded Urysohn $1$-width: it
admits a continuous map to a graph whose fibers have uniformly bounded
diameter.

Gromov \cite[\S3(b)]{Gromov1996} proposed an analogous geometric picture
under a positive lower bound for isotropic curvature. Recall that
\[
\UW_1(X)=\inf_{f:X\to\Gamma}\sup_{y\in\Gamma}
\operatorname{diam}_X f^{-1}(y),
\]
where the infimum is over continuous maps to graphs.

\begin{conjecture}[Gromov]\label{conj:width}
Let $(M^n,g)$ be a closed connected $\sigma$-PIC manifold, where
$\sigma>0$. Then its universal cover satisfies
\[
\UW_1(\widetilde M,g)\le C(n)\sigma^{-1/2}
\]
for a constant $C(n)$ depending only on the dimension.
\end{conjecture}

A related conjecture of Gromov and Schoen
(\cite[\S3(b)]{Gromov1996}, \cite[Conjecture E1$'$]{Schoen2007}) concerns
the size of stable minimal disks.

\begin{conjecture}[Gromov; Schoen]\label{conj:disk-radius}
Let $(M^n,g)$ be a complete $\sigma$-PIC manifold, where $\sigma>0$.
Every compact stable minimal disk $D$ in $M$ satisfies
\[
\sup_{x\in D}d_D(x,\partial D)\le C(n)\sigma^{-1/2},
\]
where $d_D$ is the intrinsic distance and $C(n)$ depends only on the
dimension.
\end{conjecture}

This radius estimate was proposed as a means of obtaining the width
bound in Conjecture~\ref{conj:width}.

Three approaches have played a central role in the study of PIC
manifolds.
\begin{enumerate}[(i)]
\item The first uses minimal two-surfaces, or harmonic maps, with
isotropic curvature entering through the second variation. The work of
Micallef--Moore, Fraser, and Fraser--Wolfson
\cite{MM1988,Fraser2003,FraserWolfson2006} gives restrictions on homotopy
groups and fundamental groups. This approach is analogous to the use of
geodesics under positive Ricci curvature and of minimal hypersurfaces in
three-manifolds with positive scalar curvature.

\item The second approach uses differential two-forms in even dimensions.
Micallef--Wang \cite{MW1993} showed that PIC makes the curvature term in
the Bochner formula positive on two-forms, yielding the vanishing of the
second Betti number; see also Seaman \cite{Seaman1993}. This is a
differential-form counterpart to the minimal-surface method, with the
Bochner identity taking the place of the stability inequality.

\item The third approach is Ricci flow. Hamilton \cite{Hamilton1997} developed
Ricci flow with surgery for PIC four-manifolds, with subsequent work of
Chen--Zhu and Chen--Tang--Zhu leading to the complete classification
\cite{ChenZhu2006,CTZ2012}. Preservation of PIC in all dimensions was
proved independently by Brendle--Schoen and Nguyen
\cite{BrendleSchoen2009,Nguyen2010}. Brendle \cite{Brendle2019} then
established curvature pinching estimates and a surgery theory in
dimensions $n\ge12$, obtaining classification under an incompressibility
hypothesis. As in Hamilton's work in dimension three, preservation and
improvement of curvature under the flow are the essential mechanisms.
\end{enumerate}

Recent results have advanced all three approaches. Fraser--Schoen
\cite{FraserSchoen2025} obtained a quantitative systolic bound for stable
minimal tori under a positive isotropic curvature lower bound. In
\cite{ChowWang}, the authors proved the sharp Micallef--Moore index
bound $n-3$ for nonconstant harmonic two-spheres. Combining this estimate with recent topological and smoothing
results, they proved the finite-cover diffeomorphism conjecture in
dimensions five and six. Using a twisted de Rham--Hodge operator, the first author and Zhu
\cite{ChowZhu2024} established bandwidth and focal-radius estimates under
additional curvature, boundary, and topological hypotheses. In Ricci
flow, Chen \cite{Chen2026} extended Brendle's pinching estimates to
$n\ge9$, and Cho \cite{Cho2026} subsequently treated dimension eight.
Huang \cite{Huang2025} also removed Brendle's additional assumption in higher dimensions.

Our results are twofold. First, we show that the proposed width and
stable-disk radius bounds fail even on spheres.

\begin{maintheorem}\label{thm:flexibility}
For every $n\ge4$ and every $L>0$, there exist a smooth $4$-PIC metric
$g$ on $S^n$ and an embedded compact stable minimal disk
$D\subset(S^n,g)$ such that
\[
\UW_1(S^n,g)\ge L,
\qquad
\sup_{x\in D}d_D(x,\partial D)\ge L.
\]
\end{maintheorem}

These examples disprove the geometric conjectures while remaining
consistent with the proposed topological classification. Their
construction is motivated by the equality case of a sharp eigenvalue
estimate for the Hodge Laplacian on two-forms. The isotropic curvature lower bound controls the eigenvalue, and
equality forces roundness if a closed two-form satisfying the equality
equations has rank at least four at some point. The rank-two equality equations
suggest the surface models used in the construction.

Write $\Delta_H=dd^*+d^*d$, and let $\lambda_1^{(2)}$ denote its lowest
eigenvalue on two-forms. On a closed even-dimensional PIC manifold,
there are no nonzero harmonic two-forms, so this eigenvalue is positive.
For the unit round sphere and the standard product $S^{n-1}\times S^1$,
the lowest two-form eigenvalues are $2(n-1)$ and $n-1$, respectively,
so both attain the bound below with $\sigma=4$ and $\sigma=2$.

\begin{maintheorem}\label{thm:spectral-rigidity}
Let $(M^n,g)$ be a closed connected $\sigma$-PIC manifold of even
dimension $n\ge4$, where $\sigma>0$. Then
\begin{equation}\label{eq:main-eigenvalue-estimate}
\lambda_1^{(2)}(M,g)\ge\frac{n-1}{2}\sigma.
\end{equation}
If equality holds, the universal cover admits a nonzero closed two-form
$\omega$ satisfying
\[
\Delta_H\omega=\frac{n-1}{2}\sigma\,\omega,
\qquad
\nabla_X\omega=-\frac{1}{n-1}X^\flat\wedge d^*\omega.
\]
If $\rank\omega\ge4$ at some point, then $(M,g)$ is isometric to the
round sphere or real projective space of sectional curvature $\sigma/4$.
\end{maintheorem}

The spectral estimate combines the Bochner curvature bound
\cite{ChowZhu2024} with the twistor decomposition \cite{Semmelmann2003},
following Gallot--Meyer \cite{GallotMeyer1975}. In higher rank, the
equality equations force a special Killing one-form and hence roundness.

For a closed connected manifold with $\Ric\ge(n-1)g$, the
Lichnerowicz--Obata theorem gives $\lambda_1\ge n$ for the
first nonzero eigenvalue of the Laplacian, with equality precisely for the unit round sphere
\cite{Lichnerowicz1958,Obata1962}.
For scalar curvature, Friedrich's estimate on a closed spin manifold is $\mu^2\ge\frac{n}{4(n-1)}\min_M\operatorname{Scal}_g$
for every eigenvalue $\mu$ of the spin Dirac operator
\cite{Friedrich1980}. When the scalar curvature is positive, equality
is characterized by a real Killing spinor; the metric is then Einstein
but need not be round \cite{Baer1993}.
For the PIC estimate, our rigidity conclusion requires rank at least
four. In rank two, the equality equations admit conformal product
models with a spherical factor and reduce to a scalar equation on
a surface. These models motivate the geometric counterexamples.

Section~\ref{sec:spectral} proves the spectral estimate and the
existence of a closed equality form in Theorem~\ref{thm:spectral-rigidity}.
Section~\ref{sec:construction} develops the surface construction for
Theorem~\ref{thm:flexibility}.
Section~\ref{sec:rigidity} proves the higher-rank rigidity assertion.

\vspace{0.5cm}

\section{The sharp spectral estimate}\label{sec:spectral}
Let $n\ge4$ be an even integer.
\subsection{The Weitzenb\"ock curvature term under PIC}

Let $(V,\ip{\cdot}{\cdot})$ be an oriented Euclidean vector space of even
dimension $n=2m$.  Let $R$ be an algebraic curvature tensor satisfying the
$\sigma$-PIC condition; the results of this subsection are purely algebraic and will be applied pointwise, with $V=T_xM$ and $R=R_x$.  We write $\Q$ for the
Bochner curvature endomorphism on $\Lambda^2V^*$. Let $\{e_1,\cdots,e_n\}$ denote an orthonormal basis, and let $\{e^1,\cdots,e^n\}$ denote the dual basis. With the convention
$\Q = e^j\wedge\iota_{e_i}R(e_i,e_j)$, the Bochner formula is
\[
        \Delta_H\omega=\nabla^*\nabla\omega+\Q(\omega).
\]
At a point, we choose the orthonormal basis so that \(\omega\) is in skew-normal form:
\begin{equation}\label{eq:skew-normal-form}
        \omega=\sum_{a=1}^{m}\lambda_a e^{2a-1}\wedge e^{2a}.
\end{equation}
Set $p_a=2a-1$ and $q_a=2a$.  For $a\ne b$ define
\begin{align*}
        A_{ab}:={}&R_{p_ap_bp_ap_b}+R_{q_ap_bq_ap_b}
        +R_{p_aq_bp_aq_b}+R_{q_aq_bq_aq_b},\\
        B_{ab}:={}&R_{p_aq_ap_bq_b}.
\end{align*}
Applying the PIC inequality to the frames
$(e_{p_a},e_{q_a},e_{p_b},e_{q_b})$ and
$(e_{p_a},e_{q_a},e_{p_b},-e_{q_b})$ gives
\begin{equation}\label{eq:AminusB}
        A_{ab}-2\norm{B_{ab}}\ge \sigma.
\end{equation}

The nonnegativity $\ip{\Q(\omega)}{\omega}\ge 0$ under PIC is due to
Micallef--Wang~\cite{MW1993} (see also Seaman~\cite{Seaman1993});
the quantitative lower bound below was also obtained in
\cite{ChowZhu2024}.

\begin{lemma}\label{lem:curvature-lower-bound}
If $n=2m$ and $R$ is $\sigma$-PIC, then
\begin{equation}\label{eq:curvature-lower-bound}
        \ip{\Q(\omega)}{\omega}\ge \frac{n-2}{2}\sigma\norm{\omega}^2
\end{equation}
for every $\omega\in\Lambda^2V^*$.
Moreover, in an adapted frame \eqref{eq:skew-normal-form}, equality in
\eqref{eq:curvature-lower-bound} implies the following pairwise conditions:
\begin{enumerate}[label=\textup{(\roman*)},leftmargin=2.2em]
	\item If $\lambda_a\ne0$, then for every $b\ne a$,
		\begin{equation*}
       			 A_{ab}-2\norm{B_{ab}}=\sigma.
		\end{equation*}
	\item If $\norm{\lambda_a}\ne\norm{\lambda_b}$, then $B_{ab}=0$.
\end{enumerate}
\end{lemma}

\begin{proof}
For the curvature term on two-forms one has, in the frame
\eqref{eq:skew-normal-form},
\begin{align}\label{eq:curv-expansion}
\ip{\Q(\omega)}{\omega}
={}&\sum_{1\le a<b\le m}(\lambda_a^2+\lambda_b^2)(A_{ab}-2\norm{B_{ab}})\notag\\
&+4\sum_{1\le a<b\le m}(\norm{B_{ab}\lambda_a\lambda_b}-B_{ab}\lambda_a\lambda_b)\notag\\
&+2\sum_{1\le a<b\le m}\norm{B_{ab}}(\norm{\lambda_a}-\norm{\lambda_b})^2.
\end{align}
This is obtained by expanding the Weitzenb\"ock term in the skew-normal frame and
adding the elementary identity
\[
0=2\norm{B_{ab}}(\norm{\lambda_a}-\norm{\lambda_b})^2
  +4\norm{B_{ab}\lambda_a\lambda_b}
  -2\norm{B_{ab}}(\lambda_a^2+\lambda_b^2)
\]
for each pair $a<b$.  The last two sums in \eqref{eq:curv-expansion} are
nonnegative, and \eqref{eq:AminusB} gives
\[
        \ip{\Q(\omega)}{\omega}
        \ge \sigma\sum_{a<b}(\lambda_a^2+\lambda_b^2)
        =\sigma(m-1)\sum_a\lambda_a^2
        =\frac{n-2}{2}\sigma\norm{\omega}^2.
\]
Subtracting \(\sigma\sum_{a<b}(\lambda_a^2+\lambda_b^2)\) from \eqref{eq:curv-expansion}, the equality conditions follow from the vanishing of each nonnegative summand.
\end{proof}

\subsection{The estimate and the equality equations}

Let $(M^n,g)$ be a closed Riemannian manifold.

For a two-form $\omega$, define
\begin{equation}\label{eq:P-definition}
        P_X\omega
        :=
        \nabla_X\omega-\frac13\iota_X\dd\omega
        +\frac{1}{n-1}X^\flat\wedge\dstar\omega .
\end{equation}
We regard $P\omega$ as a section of $T^*M\otimes\Lambda^2T^*M$, with
\[
        \norm{P\omega}^2:=\sum_i\norm{P_{e_i}\omega}^2
\]
for any local orthonormal frame $e_1,\ldots,e_n$.

\begin{lemma}\label{lem:twistor-identity}
For every $\omega\in\Omega^2(M)$,
\begin{equation}\label{eq:P-decomposition}
        \norm{\nabla\omega}^2
        =
        \norm{P\omega}^2
        +\frac13\norm{\dd\omega}^2
        +\frac{1}{n-1}\norm{\dstar\omega}^2 .
\end{equation}
\end{lemma}

\begin{proof}
The three terms in the twistor decomposition
\[
\nabla_X\omega=P_X\omega+\frac13\iota_Xd\omega
-\frac{1}{n-1}X^\flat\wedge d^*\omega
\]
are orthogonal in $T^*M\otimes\Lambda^2T^*M$; see
\cite{Semmelmann2003}. Their squared norms are respectively
$|P\omega|^2$, $|d\omega|^2/3$, and $|d^*\omega|^2/(n-1)$.
\end{proof}

Combining \eqref{eq:P-decomposition} with the Bochner formula gives the identity
\begin{equation}\label{eq:key-integral-identity}
\frac23\int_M\norm{\dd\omega}^2
+
\frac{n-2}{n-1}\int_M\norm{\dstar\omega}^2
=
\int_M\norm{P\omega}^2
+
\int_M\ip{\Q(\omega)}{\omega}
\end{equation}
for every two-form $\omega$.

\begin{proposition}\label{prop:eigenvalue-estimate}
Let $(M^n,g)$ be a closed connected $\sigma$-PIC manifold of even
dimension $n\ge4$. Then the estimate \eqref{eq:main-eigenvalue-estimate}
holds; that is,
\begin{equation*}
        \lambda_1^{(2)}(M,g)\ge \frac{n-1}{2}\sigma.
\end{equation*}
\end{proposition}

\begin{proof}
Let $0\ne\omega\in\Omega^2(M)$. Since $\frac23\le \frac{n-2}{n-1}$ for $n\ge4$, by \eqref{eq:key-integral-identity} and Lemma
\ref{lem:curvature-lower-bound},
\[
\begin{aligned}
        &\frac{n-2}{n-1}\int_M\bigl(\norm{\dd\omega}^2+\norm{\dstar\omega}^2\bigr)\\
        &=
        \int_M\norm{P\omega}^2
        +
        \int_M\ip{\Q(\omega)}{\omega}
        +
        \left(\frac{n-2}{n-1} - \frac{2}{3}\right)\int_M \norm{\dd\omega}^2  \\
        &\ge
        \frac{n-2}{2}\sigma\int_M\norm{\omega}^2 .
\end{aligned}
\]
Thus we obtain
\[
        \frac{\int_M\bigl(\norm{\dd\omega}^2+\norm{\dstar\omega}^2\bigr)}
             {\int_M\norm{\omega}^2}
        \ge
        \frac{n-1}{2}\sigma .
\]
Taking the infimum over all nonzero two-forms proves Proposition~\ref{prop:eigenvalue-estimate}.
\end{proof}

\begin{proposition}
\label{prop:equality-basic}
Assume equality holds in Proposition~\ref{prop:eigenvalue-estimate}. Then there exists a nonzero pair
$(\alpha,\omega)\in\Omega^1_{\dstar}(\widetilde M)\oplus
\Omega^2_{\dd}(\widetilde M)$ on the universal cover $\pi:(\widetilde M,\widetilde g)\to(M,g)$ satisfying
\begin{equation}\label{eq:equality-basic}
        \omega=\dd\alpha,
        \quad
        \dstar\omega=\frac{n-1}{2}\sigma\,\alpha,
        \quad
        \nabla_X\omega=-\frac{\sigma}{2}X^\flat\wedge\alpha,
        \quad
        \Q(\omega)=\frac{n-2}{2}\sigma\,\omega .
\end{equation}
Here $\Omega^1_{\dstar}(\widetilde M)=\{\alpha\in\Omega^1(\widetilde M)\mid \dstar\alpha=0\}$ and $\Omega^2_{\dd}(\widetilde M)=\{\omega\in\Omega^2(\widetilde M)\mid \dd\omega=0\}$.
All operators are computed with respect to the lifted metric $\widetilde g$.
\end{proposition}

\begin{proof}
Since equality in Proposition~\ref{prop:eigenvalue-estimate} is attained,
there exists a nonzero two-form $\beta\in\Omega^2(M)$ such that $\Delta_H\beta=\frac{n-1}{2}\sigma\,\beta$.

We first construct, on the universal cover, a nonzero closed two-form $\omega$
satisfying
\begin{equation}\label{eq:closed-equality-form-upstairs}
        \Delta_H\omega=\frac{n-1}{2}\sigma\,\omega,
        \qquad
        P\omega=0 .
\end{equation}

Assume first that $n\ge6$. Equality in the proof of
Proposition~\ref{prop:eigenvalue-estimate} gives
\[
\begin{aligned}
0={}&
        \int_M\norm{P\beta}^2
        +
        \int_M
        \left(
        \ip{\Q(\beta)}{\beta}
        -\frac{n-2}{2}\sigma\norm{\beta}^2
        \right)+
        \left(\frac{n-2}{n-1}-\frac23\right)
        \int_M\norm{\dd\beta}^2 .
\end{aligned}
\]
Each term is nonnegative. Since $\frac{n-2}{n-1}-\frac23>0$ for $n\ge6$,
we get $\dd\beta=0$ and $P\beta=0$. In this case we set $\omega:=\pi^*\beta$.
Then \eqref{eq:closed-equality-form-upstairs} holds.

It remains to consider $n=4$. Then $\Delta_H\beta=\frac32\sigma\,\beta$.
Equality in the proof of Proposition~\ref{prop:eigenvalue-estimate} gives
\[
        \int_M\norm{P\beta}^2
        +
        \int_M\left(
        \ip{\Q(\beta)}{\beta}
        -\sigma\norm{\beta}^2
        \right)
        =0.
\]
Hence $P\beta=0$ pointwise.

If $\dd\dstar\beta\ne0$, we set $\gamma:=\frac{2}{3\sigma}\dd\dstar\beta$.
Then $\gamma$ is a nonzero closed two-form on $M$ and $\Delta_H\gamma=\frac32\sigma\,\gamma$,
because $\Delta_H$ commutes with $\dd$ and $\dstar$. Applying the identity
\eqref{eq:key-integral-identity} to this closed eigenform gives
\[
        \frac23\int_M\norm{\dstar\gamma}^2
        =
        \int_M\norm{P\gamma}^2+\int_M\ip{\Q(\gamma)}{\gamma}.
\]
Since
\[
        \int_M\norm{\dstar\gamma}^2
        =
        \frac32\sigma\int_M\norm{\gamma}^2
\]
and Lemma~\ref{lem:curvature-lower-bound} gives
$\ip{\Q(\gamma)}{\gamma}\ge\sigma\norm{\gamma}^2$, we get $P\gamma=0$.
In this case we set $\omega:=\pi^*\gamma$.
Then \eqref{eq:closed-equality-form-upstairs} holds.

If instead $\dd\dstar\beta=0$, then
$\dstar\dd\beta=\frac32\sigma\,\beta$ is nonzero. Applying $\dstar$ gives
$\dstar\beta=0$. Since $\widetilde M$ is simply connected, it is orientable;
fix an orientation on $\widetilde M$. We set $\omega:=\star\pi^*\beta$.
Then $\omega$ is nonzero. Since $\dstar\pi^*\beta=0$, the form $\omega$ is
closed. Since the Hodge star commutes with $\Delta_H$ on two-forms in
dimension four, we have $\Delta_H\omega=\frac32\sigma\,\omega$.
It remains only to check $P\omega=0$. Let $\widetilde\beta=\pi^*\beta$.
Since $P\widetilde\beta=0$ and $\dstar\widetilde\beta=0$, we have
\[
        \nabla_X\widetilde\beta=\frac13\iota_X\dd\widetilde\beta .
\]
Using $\omega=\star\widetilde\beta$, $\star^2=1$ on two-forms in dimension
four, and $\dstar\omega=-\star\dd\widetilde\beta$, we obtain
\[
        \nabla_X\omega
        =
        \star\nabla_X\widetilde\beta
        =
        \frac13\star\iota_X\dd\widetilde\beta
        =
        \frac13 X^\flat\wedge\star\dd\widetilde\beta
        =
        -\frac13 X^\flat\wedge\dstar\omega .
\]
Since $\dd\omega=0$ and $n-1=3$, this is exactly $P\omega=0$.
Thus \eqref{eq:closed-equality-form-upstairs} holds also in this case.

We now finish the proof on $\widetilde M$. Since $\dd\omega=0$ and $\Delta_H\omega=\frac{n-1}{2}\sigma\,\omega$, we have $\dd\dstar\omega=\frac{n-1}{2}\sigma\,\omega$. Define
\[
        \alpha:=\frac{2}{(n-1)\sigma}\dstar\omega .
\]
Then
\[
        \dstar\omega=\frac{n-1}{2}\sigma\,\alpha,
        \qquad
        \dd\alpha
        =
        \frac{2}{(n-1)\sigma}\dd\dstar\omega
        =
        \omega,
        \qquad
        \dstar\alpha=0.
\]
Since $P\omega=0$ and $\dd\omega=0$, we get
\[
        \nabla_X\omega
        =
        -\frac{1}{n-1}X^\flat\wedge\dstar\omega
        =
        -\frac{\sigma}{2}X^\flat\wedge\alpha .
\]

Finally, we prove the pointwise identity for $\Q(\omega)$. From $\nabla_X\omega=-\frac{1}{n-1}X^\flat\wedge\dstar\omega$ we obtain
\[
        \nabla^*\nabla\omega
        =
        \frac{1}{n-1}\dd\dstar\omega
        =
        \frac{\sigma}{2}\omega .
\]
Using the Bochner formula and the eigen-equation for $\omega$, we conclude
\[
        \Q(\omega)
        =
        \Delta_H\omega-\nabla^*\nabla\omega
        =
        \frac{n-1}{2}\sigma\,\omega-\frac{\sigma}{2}\omega
        =
        \frac{n-2}{2}\sigma\,\omega .
\]
Thus $(\alpha,\omega)$ satisfies \eqref{eq:equality-basic}.
\end{proof}

Propositions~\ref{prop:eigenvalue-estimate} and~\ref{prop:equality-basic}
prove the estimate and the existence of a closed equality form in
Theorem~\ref{thm:spectral-rigidity}.

The equality equations in rank two suggest a surface construction,
which we use next to study the width and stable-disk radius bounds.

\vspace{0.5cm}

\section{Geometric counterexamples}\label{sec:construction}

Let $n\ge4$ be an integer. We use a conformal change on a product of a surface with $S^{n-2}$ to construct
the example in Theorem~\ref{thm:flexibility}. The construction is motivated by the equality case of the eigenvalue estimate. In the equality case of the Ricci eigenvalue estimate, let $f$ be a
first eigenfunction. Away from its critical points, the conformally scaled metric
$|df|^{-2}g$ splits off a line.
For rank-two equality under PIC, it is therefore natural to consider
$|\omega|^{-2}g$ away from the zero set of $\omega$ and expect a splitting
into a surface $B^2$ and $S^{n-2}$. This motivates the construction below.
\subsection{The surface criterion}
We first record a general construction of metrics with positive isotropic curvature, and its relation to the
rank-two equality equations.
\begin{lemma}\label{lem:conformal-product-pic}
Let $n\ge4$, let $h$ be a metric on a surface $B$ with $K_h\geq -1$, and let $u>0$ satisfy $\nabla_h^2(u^{-1})\le u^{-1}h$ and $-\Delta_hu\ge2u^3$. Define a Riemannian metric $g$ on $B\times S^{n-2}$ by $g=u^2(h+g_{S^{n-2}})$, where $g_{S^{n-2}}$ is the round metric of sectional curvature one. For every $g$-orthonormal four-frame
$\{e_1,e_2,e_3,e_4\}$, we have
\[
R_{1313}+R_{1414}+R_{2323}+R_{2424}-2R_{1234}\ge4.
\]
\end{lemma}
\begin{proof}
Let $\bar g=h+g_{S^{n-2}}$ be the product metric on $B\times S^{n-2}$ and we use $\bar R$ to denote the curvature tensor of the metric $\bar g$. Let $p\in B\times S^{n-2}$, and given a $g$-orthonormal frame $\{e_1,e_2,e_3,e_4\}$ at $p$, we let $\bar e_a=u\,e_a$ so that $\{\bar e_1,\cdots,\bar e_4\}$ is a  $\bar g$-orthonormal frame at $p$. Given any $X\in T_p(B\times S^{n-2})$, we will denote $X^H$ and $X^V$ to be the corresponding projection of $X$ onto the tangent space of $B$ and $S^{n-2}$ respectively. We then compute the curvature tensor for $\bar g$ so that
\begin{align*}
\bar{R}_{abcd}
={}&K_h\bigl(\langle \bar e_a^H,\bar e_c^H\rangle\langle \bar e_b^H,\bar e_d^H\rangle
-\langle \bar e_a^H,\bar e_d^H\rangle\langle \bar e_b^H,\bar e_c^H\rangle\bigr)\\
&+\langle \bar e_a^V,\bar e_c^V\rangle\langle \bar e_b^V,\bar e_d^V\rangle
-\langle \bar e_a^V,\bar e_d^V\rangle\langle \bar e_b^V, \bar e_c^V\rangle.
\end{align*}
Using the identity $\langle \bar e_a^V,\bar e_b^V\rangle=\delta_{ab}-\langle \bar e_a^H,\bar e_b^H\rangle$, we have
\begin{align*}
\bar{R}_{abab}
&=1-|\bar e_a^H|^2-|\bar e_b^H|^2+(K_h+1)|\bar e_a^H\wedge \bar e_b^H|^2,\\
\bar{R}_{1234}
&=(K_h+1)\bigl(\langle \bar e_1^H,\bar e_3^H\rangle\langle \bar e_2^H,\bar e_4^H\rangle
-\langle \bar e_1^H,\bar e_4^H\rangle\langle \bar e_2^H,\bar e_3^H\rangle\bigr).
\end{align*}
Consequently, 
\begin{align*}
&\bar{R}_{1313}+\bar{R}_{1414}+\bar{R}_{2323}
+\bar{R}_{2424}-2\bar{R}_{1234}\\
&\quad=4-2\sum_{a=1}^4|\bar e_a^H|^2
+(K_h+1)\bigl(
|\bar e_1^H\wedge \bar e_3^H-\bar e_2^H\wedge \bar e_4^H|^2
+|\bar e_1^H\wedge \bar e_4^H+\bar e_2^H\wedge \bar e_3^H|^2\bigr)\\
&\quad\geq 4-2\sum_{a=1}^4|\bar e_a^H|^2.
\end{align*}
Next, we let $w=u^{-1}$ so that $g=w^{-2}\bar g$, then we compute the Riemann curvature of $g$ by
\begin{align*}
R_{abcd}
={}&w^2\bar{R}_{abcd}
+w\bigl(\nabla^2_{\bar g}w(\bar e_a,\bar e_c)\delta_{bd}
+\nabla^2_{\bar g}w(\bar e_b,\bar e_d)\delta_{ac}\\
&-\nabla^2_{\bar g}w(\bar e_a,\bar e_d)\delta_{bc}
-\nabla^2_{\bar g}w(\bar e_b,\bar e_c)\delta_{ad}\bigr)
-|\nabla_{\bar g}w|^2(\delta_{ac}\delta_{bd}-\delta_{ad}\delta_{bc})\\
={}&w^2\bar{R}_{abcd}
+w\bigl(\nabla^2_{h}w(\bar e_a^H,\bar e_c^H)\delta_{bd}
+\nabla^2_{h}w(\bar e_b^H,\bar e_d^H)\delta_{ac}\\
&-\nabla^2_{h}w(\bar e_a^H,\bar e_d^H)\delta_{bc}
-\nabla^2_{h}w(\bar e_b^H,\bar e_c^H)\delta_{ad}\bigr)
-|\nabla_{h}w|^2(\delta_{ac}\delta_{bd}-\delta_{ad}\delta_{bc})
\end{align*}
This implies
\begin{align*}
&R_{1313}+R_{1414}+R_{2323}+R_{2424}-2R_{1234}\\
&\quad=w^2\bigl(\bar{R}_{1313}+\bar{R}_{1414}
+\bar{R}_{2323}+\bar{R}_{2424}-2\bar{R}_{1234}\bigr)
+2w\sum_{a=1}^4\nabla_h^2w(\bar e_a^H,\bar e_a^H)-4|\nabla_hw|^2\\
&\quad\ge4w^2-4|\nabla_hw|^2
-2w\sum_{a=1}^4(wh-\nabla_h^2w)(\bar e_a^H,\bar e_a^H).
\end{align*}
Choose an $h$-orthonormal eigenbasis $v_1,v_2$ of
$wh-\nabla_h^2w$ with eigenvalues $\lambda_1,\lambda_2\ge0$.
It then follows that
\[
\sum_{a=1}^4(wh-\nabla_h^2w)(\bar e_a^H,\bar e_a^H)
=\lambda_1\sum_{a=1}^4\langle \bar e_a,v_1\rangle^2+\lambda_2\sum_{a=1}^4\langle \bar e_a,v_2\rangle^2
\le\lambda_1+\lambda_2=2w-\Delta_hw.
\]
Thus
\begin{align*}
&R_{1313}+R_{1414}+R_{2323}+R_{2424}-2R_{1234}\\
&\quad\ge4w^2-4|\nabla_hw|^2-2w(2w-\Delta_hw)\\
&\quad=2u^{-1}\Delta_h(u^{-1})-4|\nabla_h(u^{-1})|^2\\
&\quad=-2u^{-3}\Delta_hu\ge4.
\end{align*}
This now completes the proof.
\end{proof}

\begin{remark}\label{rem:rank-two-equations}
In Lemma~\ref{lem:conformal-product-pic}, assume moreover that
$-\Delta_hu=2u^3$ and that $B$ is oriented. Set
\[
\omega=u^3\,d\vol_h,\qquad \alpha=-\frac12*_hdu.
\]
Then $\omega$ has rank two and
\[
d\alpha=\omega,\qquad d^*\omega=2(n-1)\alpha,\qquad
\nabla_X\omega=-2X^\flat\wedge\alpha.
\]
In particular, $d\omega=0$ and $\Delta_H\omega=2(n-1)\omega$.
Thus these forms satisfy the rank-two equality equations in
Theorem~\ref{thm:spectral-rigidity} with $\sigma=4$.
Indeed, the warped-product connection gives
$\nabla_X\omega=X^\flat\wedge *_hdu$; the remaining identities
follow by contraction and $d(*_hdu)=(\Delta_hu)d\vol_h$.
\end{remark}

\subsection{Construction and proof of Theorem~\ref{thm:flexibility}}

In this section, we construct the example for Theorem \ref{thm:flexibility}. We fix a sufficiently large $L>0$. We then set $r_L=\sinh(2\sqrt6\,L)$ and $b_L(r)=1-\frac{\sinh^2r}{\sinh^2r_L}$. On the disk $B=\{x\in\mathbb R^2:|x|<r_L\}$, we define a Riemannian metric 
\[
{h_L=\coth^2r_L\left(b_L(r)^{-2}\,dr^2
+b_L(r)^{-1}\sinh^2r\,d\theta^2\right)}
\]
and a function $u_L^2=\frac{\tanh^2r_L}{6(1+r^2)}\,b_L(r)$.
We then consider the metric on $B\times S^{n-2}$ defined by
\(
g_L=u_L^2(h_L+g_{S^{n-2}}).
\)
We first verify the assumptions of Lemma~\ref{lem:conformal-product-pic}.
It is straightforward to check that the Gaussian curvature of $h_L$ is $-1$ at each point of $B$, and the eigenvalues of the symmetric $(0,2)$-tensor
$u_L^{-1}h_L-\nabla_{h_L}^2(u_L^{-1})$ with respect to $h_L$ are given by $6u_Lb_L(r)\frac{r^2(r^2+2)}{1+r^2}$ and $6u_L(1+r^2-r\coth r)$. Since $r\coth r\le1+r^2$, both eigenvalues are nonnegative, so $u_L^{-1}h_L-\nabla_{h_L}^2(u_L^{-1})\ge0$. Moreover, we compute
\begin{align*}
-\frac{\Delta_{h_L}u_L}{6u_L^3}&=r\coth r-2+\frac3{1+r^2}\\
&+\frac{2\big((1+r^2)\cosh r-r\sinh r\big)^2
+r^2(r^2+2)\sinh^2r}{(1+r^2)\sinh^2r_L}\\
&\ge r\coth r-2+\frac3{1+r^2}.
\end{align*}
Moreover, we note that $r\coth r-2+\frac3{1+r^2}\geq \frac{1}{2}$ for all $0\le r\le1$,
and for $r\ge 1$ we have $r\coth r\ge r$ and
\[
r-2+\frac3{1+r^2}-\frac13
=\frac{3(r-1)^3+2(r-3/2)^2+1/2}{3(1+r^2)}>0.
\]
Thus we have
$-\Delta_{h_L}u_L\ge2u_L^3$, and Lemma~\ref{lem:conformal-product-pic}
implies that $g_L$ is $4$-PIC.

Next, we show that $g_L$ extends smoothly to a Riemannian metric on $S^n$. We first write
\[
g_L=\frac1{6(1+r^2)}\Big(
\frac{dr^2}{b_L(r)}+\sinh^2r\,d\theta^2
+\tanh^2r_L\,b_L(r)g_{S^{n-2}}\Big).
\]
At $r=0$, the functions $b_L(r)$ and $(\sinh r/r)^2$
are smooth functions of $r^2$ with value one, so the
$S^1$ factor collapses smoothly. At $r=r_L$, we have $b_L(r_L)=0$ and $b_L'(r_L)=-2\coth r_L.$ Hence, the $S^{n-2}$ factor collapses smoothly at $r=r_L$. Thus, $g_L$ extends to a smooth $4$-PIC metric on
$S^1*S^{n-2}\cong S^n$. The $4$-PIC bound extends
to both collapsed orbits by continuity.

Let $\rho=\operatorname{arsinh}r$ and
$\rho_L=\operatorname{arsinh}r_L=2\sqrt6\,L$.
Fix $p\in S^{n-2}$ and define
\[
D_L=\{[r,\theta,p]:0\le\rho\le3\rho_L/4\}.
\]
This is a smooth embedded compact disk contained in $B\times\{p\}$,
with induced metric
\[
g_{D,L}=u_L^2h_L
=\frac1{6(1+r^2)}\Big(\frac{dr^2}{b_L(r)}
+\sinh^2r\,d\theta^2\Big).
\]
The surface $B\times\{p\}$ is totally geodesic in $(S^n,g_L)$.
Choose a $g_{S^{n-2}}$-orthonormal basis $\{e_a\}_{a=1}^{n-2}$
of $T_pS^{n-2}$. Then $N_a=u_L^{-1}e_a$ form a
$g_L$-orthonormal frame of the normal bundle of $D_L$,
parallel with respect to the normal connection of $g_L$. Let $E_1,E_2$ be a local $g_{D,L}$-orthonormal frame. Then
\[
\sum_{i=1}^2R(E_i,N_a,E_i,N_b)
=-\frac{\Delta_{g_{D,L}}u_L}{u_L}\delta_{ab}.
\]
Given a normal variation $V=\sum_{a=1}^{n-2} f_a\,N_a$ that vanishes on $\partial D_L$, the
divergence theorem gives
\[
\begin{aligned}
I(V,V)
&=\sum_a\int_{D_L}\Big(
|\nabla f_a|_{g_{D,L}}^2+
\frac{\Delta_{g_{D,L}}u_L}{u_L}f_a^2\Big)d\vol_{g_{D,L}}\\
&=\sum_a\int_{D_L}u_L^2\,
\Big|\nabla\Big(\frac{f_a}{u_L}\Big)\Big|_{g_{D,L}}^2
 d\vol_{g_{D,L}}\ge0.
\end{aligned}
\]
Hence, $D_L$ is a stable minimal surface.

Since $0<b_L\le1$, $\frac{d\rho}{dr}=\frac1{\sqrt{1+r^2}}$, and
$\sinh^2r\ge r^2+\frac{r^4}{3}$ implies
$\frac{\sinh r}{\sqrt{1+r^2}}\ge\frac r2\ge\frac\rho2$, we obtain
\[
g_{D,L}\ge\frac16\,d\rho^2+\frac1{24}\rho^2d\theta^2.
\]
Consequently,
\[
\sup_{x\in D_L}d_{D_L}(x,\partial D_L)
\ge d_{D_L}(0,\partial D_L)
\ge\frac{3\rho_L}{4\sqrt6}=\frac32L.
\]
This proves the second inequality in Theorem~\ref{thm:flexibility}.

Next, we estimate the Urysohn $1$-width from below. The map
\[
P:S^n\longrightarrow\{[r,\theta,p]:0\le r\le r_L\},\qquad
P([r,\theta,z])=[r,\theta,p],
\]
is a continuous retraction. The formula for $g_L$ shows that
$P$ is distance non-increasing, so intrinsic and ambient
distances agree on $\{[r,\theta,p]:0\le r\le r_L\}$. Next, we define $F:\{[r,\theta,p]:0\le r\le r_L\}
\rightarrow \mathbb R^2$ by
\[F([r,\theta,p])=(\rho\cos\theta,\rho\sin\theta).
\]
Then $F$ is a homeomorphism onto the Euclidean disk of radius $\rho_L$.
The same metric estimate holds on the full horizontal disk and gives $|dF(\xi)|^2\le24|\xi|_{g_L}^2$ for tangent vectors $\xi$.
In particular, this implies
\(
|F(x)-F(y)|\le2\sqrt6\,d_{S^n}(x,y)
\)
for all points on $\{[r,\theta,p]:0\le r\le r_L\}$. We then write $Q=[-\rho_L/2,\rho_L/2]^2$ so that $Q\subset \bar B_{\rho_L}(0)$. By the Lebesgue covering lemma (see \cite{Guth2017}), we have $\UW_1(Q)\ge\rho_L$. Since $F$ restricts to a $\sqrt{24}$-Lipschitz homeomorphism
from $F^{-1}(Q)$ onto $Q$, we obtain
\[
\UW_1(S^n,g_L)
\ge \UW_1(F^{-1}(Q))
\ge \frac{\UW_1(Q)}{\sqrt{24}}
\ge \frac{\rho_L}{\sqrt{24}}
=L,
\]
where $F^{-1}(Q)$ carries the distance induced from $(S^n,g_L)$.
This completes the proof.

\vspace{0.5cm}

\section{Higher-rank rigidity}\label{sec:rigidity}

We prove $\nabla_X\alpha=\frac12\iota_X\omega$ whenever
$\rank\omega\ge4$ at some point. Together with the equality
equations, this makes $\alpha$ a special Killing one-form. We first
derive the algebraic identities needed in the full-rank case, then
use a local conformal splitting for intermediate rank.

\subsection{Curvature identities}

Throughout this section, we assume that equality holds in
Theorem~\ref{thm:spectral-rigidity}. We work on the universal cover $\pi:\widetilde M\to M$, equipped
with the lifted metric, also denoted by $g$.  Proposition~\ref{prop:equality-basic} implies that there is a nonzero pair $(\alpha,\omega)\in\Omega^1_{\dstar}(\widetilde M)\oplus
\Omega^2_{\dd}(\widetilde M)$ satisfying
\[
        \omega=\dd\alpha,
        \quad
        \dstar\omega=\frac{n-1}{2}\sigma\,\alpha,
        \quad
        \nabla_X\omega=-\frac{\sigma}{2}X^\flat\wedge\alpha,
        \quad
        \Q(\omega)=\frac{n-2}{2}\sigma\,\omega .
\]

Differentiating the equality equation $\nabla_X\omega=-\frac{\sigma}{2}X^\flat\wedge\alpha$, we get the identity
\begin{equation}\label{eq:curv-alpha-omega}
        R(X,Y)\omega
        =
        \frac{\sigma}{2}
        \bigl(
        X^\flat\wedge\nabla_Y\alpha
        -
        Y^\flat\wedge\nabla_X\alpha
        \bigr).
\end{equation}
Following \cite{Semmelmann2001}, given a vector field $X$, we may define $R^+(X):\Lambda^pT^*\widetilde M\to\Lambda^{p+1}T^*\widetilde M $ and $R^-(X):\Lambda^pT^*\widetilde M\to\Lambda^{p-1}T^*\widetilde M $ by
\[
        R^+(X)\omega:=\sum_j e_j^\flat\wedge R(X,e_j)\omega,
        \qquad
        R^-(X)\omega:=\sum_j\iota_{e_j}R(X,e_j)\omega .
\]

\begin{lemma}\label{lem:Rpm-equality}
We have\begin{align}
        R^+(X)\omega
        &=
        -\frac{\sigma}{2}X^\flat\wedge\omega,
        \label{eq:Rplus}\\
        R^-(X)\omega
        &=
        -\frac{n-2}{2}\sigma\nabla_X\alpha
        =
        -\frac{n-2}{n-1}\nabla_X\dstar\omega .
        \label{eq:Rminus}
\end{align}
\end{lemma}

\begin{proof}
Using \eqref{eq:curv-alpha-omega}, we compute
\[
\begin{aligned}
        R^+(X)\omega
        &=
        \frac{\sigma}{2}
        \sum_j e_j^\flat\wedge
        \bigl(
        X^\flat\wedge\nabla_{e_j}\alpha
        -
        e_j^\flat\wedge\nabla_X\alpha
        \bigr) \\
        &=
        -\frac{\sigma}{2}
        X^\flat\wedge
        \sum_j e_j^\flat\wedge\nabla_{e_j}\alpha
        =
        -\frac{\sigma}{2}X^\flat\wedge\dd\alpha
        =
        -\frac{\sigma}{2}X^\flat\wedge\omega .
\end{aligned}
\]
This gives \eqref{eq:Rplus}. Similarly,
\[
\begin{aligned}
        R^-(X)\omega
        &=
        \frac{\sigma}{2}
        \sum_j
        \iota_{e_j}
        \bigl(
        X^\flat\wedge\nabla_{e_j}\alpha
        -
        e_j^\flat\wedge\nabla_X\alpha
        \bigr).
\end{aligned}
\]
The first sum is
\[
        \sum_j\iota_{e_j}
        \bigl(X^\flat\wedge\nabla_{e_j}\alpha\bigr)
        =
        \nabla_X\alpha+X^\flat\,\dstar\alpha
        =
        \nabla_X\alpha,
\]
while the second sum is
\[
        \sum_j\iota_{e_j}
        \bigl(e_j^\flat\wedge\nabla_X\alpha\bigr)
        =
        (n-1)\nabla_X\alpha .
\]
Therefore
\[
        R^-(X)\omega
        =
        -\frac{n-2}{2}\sigma\nabla_X\alpha
        =
        -\frac{n-2}{n-1}\nabla_X\dstar\omega.
\]
\end{proof}

With our convention, the curvature operator $\R:\Lambda^2T_x\widetilde M\to\Lambda^2T_x\widetilde M$ is characterized by
\begin{equation}\label{eq:K-convention}
\ip{\R(X\wedge Y)}{Z\wedge W}
=
R(X,Y,Z,W).
\end{equation}
We use the metric to identify vectors with one-forms and bivectors with two-forms.  In particular, a two-form $\eta$ is identified with the skew endomorphism $\eta^\sharp$ defined by
\[
\ip{\eta^\sharp X}{Y}=\eta(X,Y).
\]
We denote $J:=\omega^\sharp$, thus
\begin{equation}\label{eq:J-def}
\ip{JX}{Y}=\omega(X,Y),\quad\text{and}\quad \iota_X\omega=(JX)^\flat.
\end{equation}
We use the same notation for the skew endomorphism $J$ and its derivation action on $\Lambda^2T_x\widetilde M$,
\[
J(X\wedge Y)=JX\wedge Y+X\wedge JY.
\]
Under the above identification of two-forms with skew endomorphisms, this action satisfies
\begin{equation}\label{eq:derivation-commutator}
(J\eta)^\sharp=[J,\eta^\sharp].
\end{equation}

\begin{lemma}
\label{lem:semmelmann-equality-identities}
For every $x\in\widetilde M$ and every $X,Y\in T_x\widetilde M$, we have
\begin{align}
        &\left(\Ric\circ J+J\circ\Ric-2(\R\omega)^\sharp
        -\frac{n-2}{2}\sigma J\right) (X\wedge Y) =0,
        \label{eq:semm0}\\
        &\bigl(J\circ\R+\R\circ J-\frac{\sigma}{2}J\bigr)(X\wedge Y)
        =0,
        \label{eq:semm1}\\
        &(\Ric\circ J-J\circ\Ric)X
        =
        -(n-2)\sigma
        \left(
        \nabla_X\alpha-\frac12\iota_X\omega
        \right).
        \label{eq:semm2}
\end{align}
Here $(\R\omega)^\sharp$ is extended to $\Lambda^2T_x\widetilde M$ by
$(\R\omega)^\sharp(X\wedge Y)=(\R\omega)^\sharp X\wedge Y
+X\wedge(\R\omega)^\sharp Y$.
\end{lemma}

\begin{proof}
We apply \cite[Sec.~7.2, Lemma~7.2.2]{Semmelmann2001}
to the two-form $\omega$. First, \cite[Lemma~7.2.2(2)]{Semmelmann2001} gives
\[
        (Y^\flat\wedge\iota_X-X^\flat\wedge\iota_Y)\Q(\omega)
        =
        -
        \bigl(
        \Ric\circ J+J\circ\Ric-2(\R\omega)^\sharp
        \bigr)(X\wedge Y),
\]
where the skew endomorphism in parentheses is extended to $\Lambda^2T_x\widetilde M$
as a derivation. Since $\Q(\omega)=\frac{n-2}{2}\sigma\,\omega$, this becomes
\[
        -
        \frac{n-2}{2}\sigma J(X\wedge Y)
        =
        -
        \bigl(
        \Ric\circ J+J\circ\Ric-2(\R\omega)^\sharp
        \bigr)(X\wedge Y).
\]
This gives \eqref{eq:semm0}.

Next, \cite[Sec.~7.2, Lemma~7.2.2(3)]{Semmelmann2001} gives
\[
        \iota_XR^+(Y)\omega-\iota_YR^+(X)\omega
        =
        -
        (J\circ\R+\R\circ J)(X\wedge Y).
\]
Using \eqref{eq:Rplus}, we compute
\[
\begin{aligned}
        \iota_XR^+(Y)\omega-\iota_YR^+(X)\omega
        &=
        -\frac{\sigma}{2}
        \iota_X(Y^\flat\wedge\omega)
        +
        \frac{\sigma}{2}
        \iota_Y(X^\flat\wedge\omega)  \\
        &=
        -\frac{\sigma}{2}
        (X^\flat\wedge\iota_Y-Y^\flat\wedge\iota_X)\omega .
\end{aligned}
\]
Under the metric identification of two-forms and bivectors,
\[
        (X^\flat\wedge\iota_Y-Y^\flat\wedge\iota_X)\omega
        =
        J(X\wedge Y).
\]
Therefore
\[
        (J\circ\R+\R\circ J)(X\wedge Y)
        =
        \frac{\sigma}{2}J(X\wedge Y),
\]
which is \eqref{eq:semm1}.

Finally, in the notation of \cite{Semmelmann2001}, $\Q=2q(R)$,
so the same lemma gives
\[
        2R^-(X)\omega+\iota_X\Q(\omega)
        =
        (\Ric\circ J-J\circ\Ric)X .
\]
Using \eqref{eq:Rminus} and
$\Q(\omega)=\frac{n-2}{2}\sigma\,\omega$, we obtain
\[
        (\Ric\circ J-J\circ\Ric)X
        =
        -\frac{2(n-2)}{n-1}\nabla_X\dstar\omega
        +
        \frac{n-2}{2}\sigma\iota_X\omega .
\]
Since $\dstar\omega=\frac{n-1}{2}\sigma\,\alpha$, this gives
\eqref{eq:semm2}.
\end{proof}

\begin{lemma}\label{lem:full-rank-algebraic-reduction}
At every $x\in\wt M$,
\begin{equation}\label{eq:Ric-commutes-J2}
        [\Ric,J^2]=0.
\end{equation}
\end{lemma}

\begin{proof}
Since $J\omega=0$, applying \eqref{eq:semm1} to $\omega$ gives $J(\R\omega)=0$.
By \eqref{eq:derivation-commutator},
\begin{equation}\label{eq:Romega-commutes-J}
        [(\R\omega)^\sharp,J]=0.
\end{equation}

On the other hand, since $n\ge4$, \eqref{eq:semm0} is equivalent to the
skew-endomorphism identity
\[
        \Ric\circ J+J\circ\Ric-2(\R\omega)^\sharp
        =
        \frac{n-2}{2}\sigma J.
\]
Taking the commutator with $J$ and using \eqref{eq:Romega-commutes-J}, we get
\[
        [\Ric\circ J+J\circ\Ric,J]=0.
\]
Since
\[
        [\Ric\circ J+J\circ\Ric,J]
        =
        \Ric\circ J^2-J^2\circ\Ric,
\]
we obtain \eqref{eq:Ric-commutes-J2}.
\end{proof}

Next, we define
\begin{equation}\label{eq:max-rank-M0}
        2r:=\max_{\wt M}\rank\omega,
        \qquad
        \wt M_0:=\{x\in\wt M:\rank\omega_x=2r\}.
\end{equation}
We call $\wt M_0$ the regular set of $\omega$. Note that $\widetilde{M}_0$ is a non-empty open subset of $\tilde{M}$. In the remainder of this section, we study the equality problem according to
the value of $2r$. All arguments are carried out on connected components of
$\wt M_0$.

\subsection{Full rank}

Assume that $2r=n$. We fix a point $x\in\wt M_0$ and choose an orthonormal basis of
$T_x\wt M$ such that
\begin{equation}\label{eq:full-rank-normal-form}
        \omega=\sum_{a=1}^r\lambda_a e^{p_a}\wedge e^{q_a},
        \qquad
        p_a=2a-1,\quad q_a=2a,\quad \lambda_a\ne0.
\end{equation}
Then $Je_{p_a}=\lambda_a e_{q_a}$ and $Je_{q_a}=-\lambda_a e_{p_a}$.

Since $\Ric$ is symmetric and $J^2$ is self-adjoint, Lemma~\ref{lem:full-rank-algebraic-reduction} implies that $\Ric$ preserves each eigenspace of $J^2$. 
The next lemma gives the curvature identities needed to prove
$[\Ric,J]=0$.

\begin{lemma}\label{lem:pairwise-diag-full}
With the notation above, for every $a\ne b$ set
\[
D_{ab}:=(R_{p_a p_b p_a p_b}+R_{p_a q_b p_a q_b})
        -(R_{q_a p_b q_a p_b}+R_{q_a q_b q_a q_b}).
\]
Then
\[
        D_{ab}=0.
\]
\end{lemma}

\begin{proof}
Fix $a\ne b$. For simplicity, we write
\[
        (e_1,e_2,e_3,e_4)=(e_{p_a},e_{q_a},e_{p_b},e_{q_b}),
        \qquad
        \lambda_1=\lambda_a,\quad \lambda_2=\lambda_b.
\]
Then $D_{ab}=(R_{1313}+R_{1414})-(R_{2323}+R_{2424})$.
The $e_1\wedge e_3$-coefficient of \eqref{eq:semm1} applied to $e_1\wedge e_4$ gives
\begin{equation}\label{eq:S1-full}
        -\lambda_1(R_{1324}-R_{2314})
        +\lambda_2(R_{1313}+R_{1414})
        =
        \frac{\sigma}{2}\lambda_2.
\end{equation}
Similarly, the $e_2\wedge e_3$-coefficient with $(X,Y)=(e_2,e_4)$ gives
\begin{equation}\label{eq:S3-full}
        -\lambda_1(R_{1324}-R_{2314})
        +\lambda_2(R_{2323}+R_{2424})
        =
        \frac{\sigma}{2}\lambda_2.
\end{equation}
Here we used the curvature symmetries, for instance
$R_{2413}=R_{1324}$, and $R_{1423}=R_{2314}$.

Subtracting
\eqref{eq:S3-full} from \eqref{eq:S1-full} gives
\[
        \lambda_2\bigl[(R_{1313}+R_{1414})-(R_{2323}+R_{2424})\bigr]=0.
\]
Since $\lambda_2\ne0$, we have $D_{ab}=0$.
\end{proof}

\begin{lemma}\label{lem:full-rank-ricci-commutes}
We have $[\Ric,J]=0$ at every $x\in\wt M_0$.
\end{lemma}

\begin{proof}
Fix $x\in\wt M_0$. By Lemma~\ref{lem:full-rank-algebraic-reduction}, $\Ric$ preserves every eigenspace of $J^2$. Let $-\mu^2$ be an eigenvalue of $J^2$, with $\mu>0$, and set
\[
        E_\mu:=\ker(J^2+\mu^2\id).
\]
In particular, we have $T_x\widetilde M=\bigoplus_{\mu>0}E_{\mu}$. It remains to prove that $\Ric$ commutes with $J$ on each fixed $E_\mu$. We set
\[
        I:=\mu^{-1}J
        \qquad\text{on }E_\mu.
\]
Then $I^2=-\id$ and $I$ is orthogonal. Let $v\in E_\mu$ be a unit vector.
Extend $v,Iv$ to an orthonormal basis adapted to $J$, with $e_1=v$ and $e_2=Iv$.

We then observe that
\begin{align*}
	\Ric(v,v)-\Ric(Iv,Iv)
        &=
        \sum_{k=1}^n(R_{1k1k}-R_{2k2k})\\
        &=
        \sum_{b=2}^r
        \Bigl[
        (R_{1p_b1p_b}+R_{1q_b1q_b})  
        -(R_{2p_b2p_b}+R_{2q_b2q_b})
        \Bigr]\\
        &= \sum_{b=2}^r D_{1b}.
\end{align*}
Therefore Lemma~\ref{lem:pairwise-diag-full} implies
\[
        \Ric(v,v)=\Ric(Iv,Iv)
\]
for every unit vector $v\in E_\mu$.
By homogeneity, this holds for every $v\in E_\mu$.
By polarization, we have $\Ric(u,v)=\Ric(Iu,Iv)$ for all $u,v\in E_\mu$.
Therefore,
\[
\begin{aligned}
        \ip{(\Ric\circ I-I\circ\Ric)u}{v}
        =
        \Ric(Iu,v)+\Ric(u,Iv)
        =0
\end{aligned}
\]
for all $u,v\in E_\mu$. Hence $[\Ric,J]=0$ on $E_\mu$.
\end{proof}

Recall that, in the terminology of \cite[Sec.~3.1]{Semmelmann2001}
(see also \cite{Semmelmann2003}), a one-form $\alpha$ is a special
Killing form with our normalization constant $c$ if
\[
        \nabla_X\alpha=\frac12\iota_X\dd\alpha,
        \qquad
        \nabla_X\dd\alpha=-2c\,X^\flat\wedge\alpha
\]
for every vector field $X$. By \eqref{eq:equality-basic},
$\dd\alpha=\omega$ and the second identity already holds with
$c=\sigma/4$.

\begin{proposition}\label{prop:full-rank-special}
On every connected component $U\subset\wt M_0$,
\[
        \nabla_X\alpha=\frac12\iota_X\omega
\]
for every vector field $X$ on $U$.
\end{proposition}

\begin{proof}
By Lemma~\ref{lem:full-rank-ricci-commutes}, $[\Ric,J]=0$ on $U$. Then \eqref{eq:semm2} gives
\[
        0=[\Ric,J]X
        =-(n-2)\sigma
        \left(\nabla_X\alpha-\frac12\iota_X\omega\right).
\]
\end{proof}

\subsection{Local splitting on the maximal-rank set}

We next assume $2r<n$ and work on a connected component $U$ of $\wt M_0$. Throughout this subsection, we set
\begin{equation}\label{eq:Theta-u}
        \Theta:=\omega^r,
        \qquad
        u:=\norm{\Theta}_g>0,
        \qquad
        K:=\ker\omega,
        \qquad
        E:=K^\perp .
\end{equation}
Since $\omega$ has constant rank $2r$ 
on $U$, both $K$ and $E$ are smooth vector bundles over $U$, with $\dim E=2r$ and $\dim K=n-2r$.

\begin{lemma}\label{lem:alpha-kernel-general}
We have
\begin{equation}\label{eq:alpha-wedge-omegar}
        \alpha\wedge\omega^r=0
\end{equation}
on $U$. Equivalently, $\alpha|_K=0$.
\end{lemma}

\begin{proof}
Since $2r$ is the maximal rank of $\omega$, we have $\omega^{r+1}\equiv0$ on $\wt M$. Differentiating and using \eqref{eq:equality-basic},
\[
\begin{aligned}
0
=
\nabla_X(\omega^{r+1})  
=
-\frac{(r+1)\sigma}{2}X^\flat\wedge\alpha\wedge\omega^r .
\end{aligned}
\]
Thus $X^\flat\wedge(\alpha\wedge\omega^r)=0$ for every vector field $X$. Since $n$ is even and $2r<n$, the form $\alpha\wedge\omega^r$ has degree $2r+1<n$. Hence $\alpha\wedge\omega^r=0$.
\end{proof}

Define $Z\in\Gamma(E)$ by
\begin{equation}\label{eq:Z-def}
        \iota_Z\omega=\alpha .
\end{equation}
Since $E=K^\perp$ and $K=\ker\omega$, the restriction $\omega|_E$ is nondegenerate. Lemma~\ref{lem:alpha-kernel-general} shows that $\alpha$ annihilates $K$, so this equation determines a unique smooth $Z$.

\begin{lemma}\label{lem:Theta-u-identities}
The following identities hold on $U$:
\begin{align}
	    &\nabla_X\Theta
        =
        -\frac{\sigma}{2}X^\flat\wedge\iota_Z\Theta,	
        \label{eq:nablaTheta-general}\\
        & \dd\log u
        =
        -\frac{\sigma}{2}Z^\flat. \label{eq:dlogu-general}     
\end{align}
In particular, $\dd u$ vanishes on $K$.
\end{lemma}

\begin{proof}
Since $\Theta=\omega^r$,
\[
\begin{aligned}
\nabla_X\Theta
=
-\frac{r\sigma}{2}X^\flat\wedge\alpha\wedge\omega^{r-1}.
\end{aligned}
\]
On the other hand, by \eqref{eq:Z-def},
\[
        \iota_Z\Theta
        =
        \iota_Z(\omega^r)
        =
        r(\iota_Z\omega)\wedge\omega^{r-1}
        =
        r\alpha\wedge\omega^{r-1}.
\]
This proves \eqref{eq:nablaTheta-general}.

Since $\Theta$ is a top-degree form on $E$, we have
\[
        \ip{X^\flat\wedge\iota_Z\Theta}{\Theta}
        =
        \ip{X}{Z}u^2 .
\]
Therefore
\[
\begin{aligned}
X(u^2)
&=
2\ip{\nabla_X\Theta}{\Theta}  \\
&=
-\sigma\ip{X^\flat\wedge\iota_Z\Theta}{\Theta}  \\
&=
-\sigma\ip{X}{Z}u^2 .
\end{aligned}
\]
Dividing by $2u^2$ gives
\[
        \dd\log u=-\frac{\sigma}{2}Z^\flat.
\]
Since $Z\in E$, this one-form vanishes on $K$.
\end{proof}

\begin{proposition}\label{prop:local-splitting}
If $2r<n$, every point of $\wt M_0$ has a neighborhood on which
$u^{-2}g$ is a product $g_B+g_F$, with $TB=E$ and $TF=K$.
The function $u$ depends only on $B$, and
\begin{equation}\label{eq:original-warped-product}
g=h+u^2g_F,\qquad h=u^2g_B.
\end{equation}
\end{proposition}

\begin{proof}
On $U$, we define
\begin{equation}\label{eq:tilde-metric-theta}
        \bar g:=u^{-2}g,
        \qquad
        \bar\Theta:=u^{-(2r+1)}\Theta.
\end{equation}
For any $p$-form $\beta$ and $\bar g=e^{2\varphi}g$, one has the identity
\[
        \bar\nabla_X\beta
        =
        \nabla_X\beta
        -p\dd\varphi(X)\beta
        -\dd\varphi\wedge\iota_X\beta
        +X^\flat\wedge\iota_{\nabla\varphi}\beta .
\]
Using \eqref{eq:nablaTheta-general} and \eqref{eq:dlogu-general}, applying the above identity to the $2r$-form $\bar\Theta=u^{-(2r+1)}\Theta$ with $\varphi=-\log u$, we thus get
\begin{align*}
\bar\nabla_X\bar\Theta
&=
u^{-(2r+1)}
\Bigl[
        \nabla_X\Theta
        -\dd\log u(X)\Theta
        +\dd\log u\wedge\iota_X\Theta
        +X^\flat\wedge\iota_{\nabla\varphi}\Theta
\Bigr]\\
&= 
u^{-(2r+1)}
\Bigl[
-\frac{\sigma}{2}X^\flat\wedge\iota_Z\Theta
+\frac{\sigma}{2}\ip{X}{Z}\Theta
-\frac{\sigma}{2}Z^\flat\wedge\iota_X\Theta
+\frac{\sigma}{2}X^\flat\wedge\iota_Z\Theta 
\Bigr]\\
&=
\frac{\sigma}{2} u^{-(2r+1)}
\bigl(\ip{X}{Z}\Theta-Z^\flat\wedge\iota_X\Theta\bigr).
\end{align*}
For a vector field $X$ on $U$, we can write $X=X_E+X_K$ with $X_E\in E$ and $X_K\in K$, then
\[
        \iota_X\Theta=\iota_{X_E}\Theta,
        \qquad
        Z^\flat\wedge\iota_X\Theta=\ip{Z}{X_E}\Theta=\ip{Z}{X}\Theta.
\]
Hence $\bar\nabla_X\bar\Theta=0$ for every $X$.

The kernel of the parallel form $\bar\Theta$ is $K$. Differentiating
$\iota_V\bar\Theta=0$ for $V\in K$ shows that $K$ is parallel for
$u^{-2}g$, and so is $E=K^\perp$. The local de Rham decomposition
gives the product. Since $du$ vanishes on $K$ by
Lemma~\ref{lem:Theta-u-identities}, $u$ depends only on the $B$-variable.
\end{proof}

\subsection{Intermediate rank}

For $2\le2r<n$, work in a local product neighborhood from
Proposition~\ref{prop:local-splitting}, with $g=h+u^2g_F$.
We first derive the trace equation, which also applies when $r=1$.
The rigidity argument then uses $r>1$.

We call vectors tangent to $B$ horizontal and vectors tangent to $F$
vertical. For arbitrary horizontal vectors $X,Y$ and vertical vectors
$V,W$, the warped-product curvature formulas give
\begin{align}
        R_g\bigl(X,u^{-1}V,Y,u^{-1}W\bigr)
        &=-\frac{\Hess_hu(X,Y)}{u}\,
        \ip{V}{W}_{g_F}, \label{eq:wp-mixed}\\
        R_g(X,Y,u^{-1}V,u^{-1}W)
        &=0. \label{eq:wp-HHVV}
\end{align}
In particular, taking $Y=X$ and $W=V$ with $|V|_{g_F}=1$ in
\eqref{eq:wp-mixed} gives
\begin{equation}\label{eq:wp-HV}
        R_g(X,u^{-1}V,X,u^{-1}V)
        =
        -\frac{\Hess_hu(X,X)}{u},
\end{equation}
while \eqref{eq:wp-mixed} vanishes whenever $V$ and $W$ are
$g_F$-orthogonal.

\begin{proposition}\label{prop:trace-equation}
On $(B^{2r},h)$, the function $u$ satisfies
\begin{equation}\label{eq:trace-equation-intermediate}
        -\Delta_hu=\frac{r\sigma}{2}u.
\end{equation}
\end{proposition}

\begin{proof}
Fix a point of the product neighborhood. Since $g|_{TB}=h$, choose an
$h$-orthonormal frame such that
\[
        \omega
        =
        \sum_{a=1}^{r}\lambda_a e^{p_a}\wedge e^{q_a},
        \qquad
        \lambda_a\ne0,\ p_a=2a-1,\ q_a=2a.
\]
Let $a\in\{1,\ldots,r\}$ be fixed. Since $n-2r\ge2$, choose any $g_F$-orthonormal vertical pair
$V,W$. Then $u^{-1}V,\, u^{-1}W\in K=\ker\omega$ are $g$-orthonormal vertical vectors. 
Complete the frame to a skew-normal frame for $\omega$ by including this vertical pair.
In this adapted frame, the coefficient of $\omega$ on the plane spanned by
$u^{-1}V$ and $u^{-1}W$ is zero.

We then apply the equality case in
Lemma~\ref{lem:curvature-lower-bound} to the pair
$e_{p_a},e_{q_a}$ with $\lambda_a\neq 0$, and the vertical pair $u^{-1}V,u^{-1}W$. Since
\[
        R_g(e_{p_a},e_{q_a},u^{-1}V,u^{-1}W)=0
\]
by \eqref{eq:wp-HHVV}, this gives
\[
\begin{aligned}
        \sigma
        ={}&
        R_g(e_{p_a},u^{-1}V,e_{p_a},u^{-1}V)
        +
        R_g(e_{q_a},u^{-1}V,e_{q_a},u^{-1}V)\\
        &+
        R_g(e_{p_a},u^{-1}W,e_{p_a},u^{-1}W)
        +
        R_g(e_{q_a},u^{-1}W,e_{q_a},u^{-1}W).
\end{aligned}
\]
Using the mixed warped-product curvature formula \eqref{eq:wp-HV} for each
term, we obtain
\[
        \Hess_hu(e_{p_a},e_{p_a})
        +
        \Hess_hu(e_{q_a},e_{q_a})
        =
        -\frac{\sigma}{2}u
\]
for each $a=1,\ldots,r$. Summing over $a=1,\ldots,r$ gives \eqref{eq:trace-equation-intermediate}.
\end{proof}

\begin{proposition}\label{prop:pure-hessian}
Suppose that $2< 2r<n$. On $(B^{2r},h)$, the function $u$ satisfies
\begin{equation}\label{eq:pure-hessian}
        \Hess_hu=-\frac{\sigma}{4}u\,h.
\end{equation}
\end{proposition}

\begin{proof}
Let $X,Y$ be any $h$-orthonormal horizontal pair. Since $2r<n$ and both $n$
and $2r$ are even, the vertical $F$-factor has dimension at least two. Choose any
$g_F$-orthonormal vertical pair $V,W$. Then $\{ X,\, Y,\, u^{-1}V,\, u^{-1}W\}$
is a $g$-orthonormal four-frame.

Applying the ambient $\sigma$-PIC condition to this frame gives
\[
\begin{aligned}
\sigma
&\le
R_g(X,u^{-1}V,X,u^{-1}V)
+
R_g(X,u^{-1}W,X,u^{-1}W)\\
&\quad
+
R_g(Y,u^{-1}V,Y,u^{-1}V)
+
R_g(Y,u^{-1}W,Y,u^{-1}W)\\
&\quad
-
2R_g(X,Y,u^{-1}V,u^{-1}W).
\end{aligned}
\]
Using the warped-product formulas \eqref{eq:wp-HV} and
\eqref{eq:wp-HHVV}, this becomes
\[
        \sigma
        \le
        -\frac{2}{u}
        \left(
        \Hess_hu(X,X)+\Hess_hu(Y,Y)
        \right).
\]
Thus, for every $h$-orthonormal horizontal pair $X,Y$,
\begin{equation}\label{eq:pairwise-hessian-ineq}
        \Hess_hu(X,X)+\Hess_hu(Y,Y)
        \le
        -\frac{\sigma}{2}u.
\end{equation}

At a point, let $\mu_1,\ldots,\mu_{2r}$ be the eigenvalues of
$\Hess_hu$. By \eqref{eq:pairwise-hessian-ineq},
$\mu_i+\mu_j\le-\sigma u/2$ for $i\ne j$. Summing over all pairs
and using Proposition~\ref{prop:trace-equation} gives equality in the
sum, so every pair inequality is an equality. Since $2r>2$, this forces
$\mu_i=-\sigma u/4$ for every $i$.
\end{proof}

\begin{proposition}
\label{prop:intermediate-special}
Suppose that $2< 2r<n$. On $\wt M_0$, we have
\begin{equation}\label{eq:special-intermediate}
        \nabla_X\alpha=\frac12\iota_X\omega
\end{equation}
for every vector field $X$.
\end{proposition}

\begin{proof}
Work in a local product neighborhood.
From \eqref{eq:dlogu-general},
\[
        Z=-\frac{2}{\sigma u}\nabla_hu.
\]
Using \eqref{eq:pure-hessian} in Proposition~\ref{prop:pure-hessian}, we compute 
\begin{equation}\label{eq:nabla-Z-intermediate}
\begin{aligned}
        \nabla_X^gZ
        =
        \nabla_X^hZ  
        &=
        -\frac{2}{\sigma}
        \left(
        \frac{1}{u}\nabla_X^h\nabla_hu
        -
        \frac{X(u)}{u^2}\nabla_hu
        \right) \\
        &=
        \frac12 X-\frac{X(u)}{u}Z .
\end{aligned}
\end{equation}
for every horizontal vector field $X$.

Since $\alpha=\iota_Z\omega$, we have $\alpha(Z)=0$.
For horizontal $X$, \eqref{eq:dlogu-general} gives
$X(u)/u=-\sigma\ip{X}{Z}/2$, so
\[
\begin{aligned}
\nabla_X\alpha
&=\iota_{\nabla_XZ}\omega+\iota_Z\nabla_X\omega\\
&=\frac12\iota_X\omega-\frac{X(u)}{u}\alpha
-\frac{\sigma}{2}\bigl(\ip{X}{Z}\alpha-\alpha(Z)X^\flat\bigr)\\
&=\frac12\iota_X\omega-\frac{X(u)}{u}\alpha
+\frac{X(u)}{u}\alpha
=\frac12\iota_X\omega.
\end{aligned}
\]

Now let $V$ be vertical. The warped-product connection gives
\[
        \nabla_VZ=\frac{Z(u)}{u}V,
\]
which is vertical. Since $V\in K$ and $\ip{V}{Z}=\alpha(Z)=0$,
\[
\begin{aligned}
\nabla_V\alpha
&=\iota_{\nabla_VZ}\omega+\iota_Z\nabla_V\omega\\
&=0-\frac{\sigma}{2}
\bigl(\ip{V}{Z}\alpha-\alpha(Z)V^\flat\bigr)
=0=\frac12\iota_V\omega.
\end{aligned}
\]

Combining the horizontal and vertical cases proves \eqref{eq:special-intermediate} on each product neighborhood.
\end{proof}

\subsection{Global rigidity}

\begin{lemma}\label{lem:density-max-rank}
$\wt M_0$ is dense in $\wt M$.
\end{lemma}

\begin{proof}
Note that $\wt M_0$ is exactly $\{x\in\wt M:\omega_x^r\ne0\}$.
Using \eqref{eq:equality-basic} and contracting in a local orthonormal frame,
\[
\begin{aligned}
\nabla_X(\omega^r)
&=-\frac{r\sigma}{2}X^\flat\wedge\alpha\wedge\omega^{r-1},\\
\dstar(\omega^r)
&=\frac{r\sigma}{2}\sum_i
\iota_{e_i}(e_i^\flat\wedge\alpha\wedge\omega^{r-1})\\
&=\frac{r\sigma}{2}(n-2r+1)\alpha\wedge\omega^{r-1},
\end{aligned}
\]
because $\alpha\wedge\omega^{r-1}$ has degree $2r-1$.
Applying $\dd$ and using $\dd\omega=0$ and $\dd\alpha=\omega$ gives
\[
        \Delta_H(\omega^r)
        =
        \frac{r\sigma}{2}(n-2r+1)\omega^r.
\]
Thus $\omega^r$ is a smooth solution of an elliptic equation, of Laplace type and hence with scalar principal symbol $\norm{\xi}^2$.
If $\wt M_0$ were not dense, then $\omega^r$ would vanish on a nonempty open
set. By weak unique continuation for elliptic systems of Laplace type \cite{Aronszajn1957,AKS1962}, this
would imply $\omega^r\equiv0$ on $\wt M$.
This contradicts the definition of $2r$ as the maximal rank of $\omega$.
Therefore $\wt M_0$ is dense.
\end{proof}

\begin{proof}[Proof of the rigidity assertion in Theorem~\ref{thm:spectral-rigidity}]
Recall that $2r=\max_{\wt M}\rank\omega$.
By assumption $2r>2$. Propositions~\ref{prop:full-rank-special} and \ref{prop:intermediate-special} give $\nabla_X\alpha=\frac12\iota_X\omega$ on $\wt M_0$.
By Lemma~\ref{lem:density-max-rank}, the set $\wt M_0$ is dense in $\wt M$.
By continuity,
\begin{equation}\label{eq:special-global}
        \nabla_X\alpha=\frac12\iota_X\omega
        \qquad\text{on }\wt M.
\end{equation}
Since $\omega=\dd\alpha$, the equality equation $\nabla_X\omega=-\frac{\sigma}{2}X^\flat\wedge\alpha$ becomes
\begin{equation}\label{eq:special-global-second}
        \nabla_X\dd\alpha
        =
        -\frac{\sigma}{2}X^\flat\wedge\alpha .
\end{equation}
On $\wt M_0$, we have $\alpha^\sharp\in(\ker\omega)^\perp$:
this is automatic in full rank and follows from
Lemma~\ref{lem:alpha-kernel-general} otherwise.
If $\norm{\alpha}^2$ were constant, then
\[
        0=X(|\alpha|^2)
        =2(\nabla_X\alpha)(\alpha^\sharp)
        =
        \omega(X,\alpha^\sharp)
\]
for every vector field $X$. Thus $\alpha^\sharp\in\ker\omega$ on $\wt M_0$, so $\alpha=0$
on this nonempty open set. Then $\omega=\dd\alpha=0$ there,
contrary to $\rank\omega=2r>0$.
Hence $\norm{\alpha}^2$ is nonconstant.

Set $\widehat g=(\sigma/4)g$. The Levi-Civita connection is unchanged and
$X^{\flat_{\widehat g}}=(\sigma/4)X^\flat$, so
\eqref{eq:special-global} and \eqref{eq:special-global-second} become
\[
\nabla_X\alpha=\frac12\iota_X\dd\alpha,\qquad
\nabla_X\dd\alpha=-2X^{\flat_{\widehat g}}\wedge\alpha.
\]
Moreover, $\norm{\alpha}_{\widehat g}^2=(4/\sigma)\norm{\alpha}_g^2$ is nonconstant.
Thus $\alpha$ is a special Killing one-form with the normalized coefficient $-2$.
The universal cover $(\wt M,\widehat g)$ is complete and simply connected.
By the nonconstant-length case of
\cite[Proposition~3.1.3]{Semmelmann2001}, $(\wt M,\widehat g)$ is the unit round sphere.
Rescaling gives
\[
(\wt M,g)\cong S^n\!\left(\frac{2}{\sqrt{\sigma}}\right).
\]

Let $\Gamma$ be the deck transformation group of
$\pi:\wt M\to M$. Since $\wt M$ is compact, $\Gamma$ is finite. Using that $n$ is even, multiplicativity of the Euler
characteristic under finite coverings gives
\[
        2=\chi(S^n)=|\Gamma|\,\chi(M).
\]
Hence $|\Gamma|=1$ or $|\Gamma|=2$. If $|\Gamma|=2$, the nontrivial deck transformation is a
fixed-point-free involutive isometry of the round sphere. Viewed as an
orthogonal involution of $\mathbb R^{n+1}$, it has no $+1$-eigenvector and
therefore equals $-\textup{Id}$. Thus $(M,g)$ is isometric either to the round
sphere or to the round real projective space of sectional curvature
$\sigma/4$. This proves the rigidity assertion in Theorem~\ref{thm:spectral-rigidity}.
\end{proof}

\end{document}